\documentclass[12pt]{amsart}
\usepackage{amssymb, graphicx}
\usepackage[mathscr]{eucal}
\usepackage{amsmath,amscd}
\usepackage{rotating}
\usepackage{epstopdf}
\usepackage{diagbox}
\usepackage{amssymb}
\usepackage{adjustbox}
\usepackage{amsfonts}
\usepackage[final]{pdfpages}
\theoremstyle{plain}
\newtheorem{prop}{Proposition}[section]
\newtheorem{coro}[prop]{Corollary}

\newtheorem{lemm}[prop]{Lemma}

\newtheorem{thm}[prop]{Theorem}
\newtheorem{ex}[prop]{Example}
\theoremstyle{definition}
\newtheorem{defn}[prop]{Definition}

\newtheorem{rem}[prop]{Remark}

\DeclareMathOperator{\breadth}{breadth}

\DeclareMathOperator{\maxdeg}{maxdeg}
\DeclareMathOperator{\mindeg}{mindeg}

\def\mcg#1;#2{\Gamma_{#1,#2}}
\def\fg#1;#2{\Pi_{#1,#2}}
\def\tb#1;#2{\mathscr{K}_{\frac{#1}{#2}}}

\begin{document}

\title[A Simple Characterization of Adequate Links]
{A Simple Characterization of Adequate Links}


\author{Khaled Qazaqzeh}
\address{Department of Mathematics, Faculty of Science, Yarmouk University, Irbid, Jordan, 21163}
\email{qazaqzeh@yu.edu.jo}
\urladdr{http://faculty.yu.edu.jo/qazaqzeh}

\author{Nafaa Chbili}
\address{Address: Department of Mathematical Sciences, College of Science, UAE University, 15551 Al Ain, U.A.E.}
\email{nafaachbili@uaeu.ac.ae}
\urladdr{http://faculty.uaeu.ac.ae/nafaachbili}

\date{17/12/2024}

\begin{abstract}
We prove that the Jones diameter of a link is twice its crossing number whenever  the breadth of its Jones polynomial equals the difference between the  crossing number and the  Turaev genus. This implies  that such  link is adequate, as per the characterization provided in \cite[Theorem\,1.1]{KL}. By combining this  with the result in \cite[Theorem\,3.2]{A}, we obtain a  characterization of adequate links  using these numerical link invariants. As an application, we provide a criterion  to obstruct a link from being  quasi-alternating. Furthermore, we  establish a lower bound for the crossing number of certain classes of links, aiding in  determining the crossing number of the link  in specific cases.
\end{abstract}

\maketitle

\section{Introduction}

Over the past few decades, the Jones polynomial has proven to be an effective tool in addressing various problems in knot theory. Notably, the independent work of Thistlethwaite \cite{Th}, Kauffman \cite{K} and Murasugi \cite{Mu} shows  that the breadth of the Jones polynomial of a link is a lower bound for its crossing number. Moreover, equality holds if and only if the link is alternating. As a result, a reduced alternating diagram has the minimal crossing number
among all diagrams of the given link. This provides a solution to one of Tait’s conjectures.\\

The class of adequate links represents a natural generalization of the well-studied class of alternating links in several aspects. This new class of links was first introduced in \cite{LT}, and it is well-known that some properties of alternating links extend, at least partially, to this class. We briefly recall the definition of adequate links in the following paragraph.\\

Given a link diagram $D$, a Kauffman state of $D$ is a choice of  $A$-resolution or $B$-resolution for every crossing of the diagram, see Figure \ref{figure}.  This choice of resolutions of the diagram $D$, denoted by $\sigma(D)$, leads to a collection of simple closed curves called state circles.  Let $|\sigma(D)|$ denote the number of these state circles.  The state $\sigma_{A}(D)$ denotes the all-$A$ state  obtained by choosing the $A$-resolution at every crossing of the diagram $D$. In a similar manner, the state $\sigma_{B}(D)$ denotes the all-$B$ state obtained by choosing the $B$-resolution at every crossing of the diagram $D$. A link diagram is said to be $A$-adequate if $|\sigma_{A}(D)| > |\sigma^{'}(D)|$ for any state $\sigma^{'}(D)$ of $D$ obtained by choosing $A$-resolution at every crossing of the diagram $D$ except one.  In a similar manner, the link is $B$-adequate if $|\sigma_{B}(D)| > |\sigma^{''}(D)|$ for any state $\sigma^{''}(D)$ of $D$ obtained by choosing $B$-resolution at every crossing of the diagram $D$ except one. A link is said to be {\it adequate} if it has a diagram that is both  $A$-adequate and $B$-adequate at the same time. If only one of the two conditions is satisfied then the link is said to be {\it semi-adequate}.

\begin{figure} [h]
\begin{center}
\includegraphics[scale=0.5]{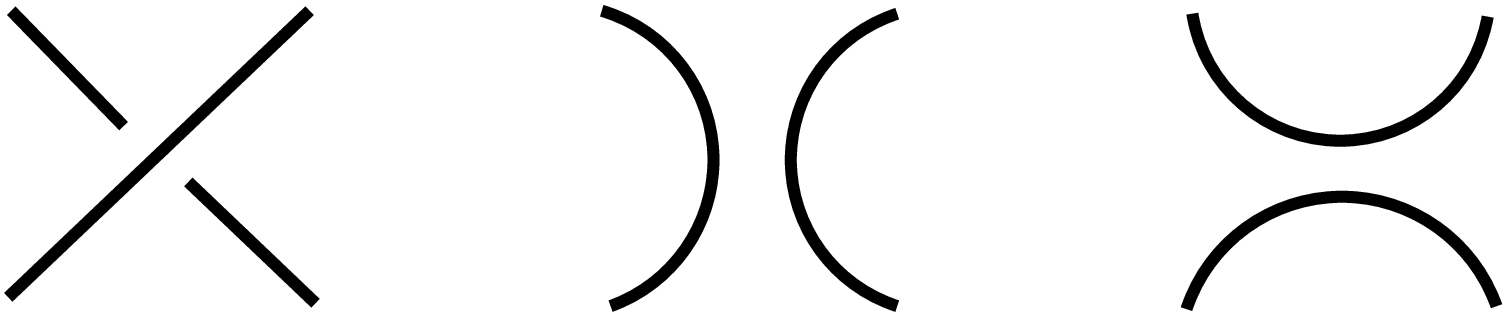} \\
\end{center}
\caption{The link diagram $L$ at the crossing $c$, its $A$-resolution and $B$-resolution respectively.}\label{figure}
\end{figure}

For the rest of this paper, we let $L$ be an oriented link, $V_{L}(t)$ be its Jones polynomial, and $J_{L}(n)(t)$ be its colored Jones polynomial. We recall the breadth of $V_L(t)$, denoted hereafter $\breadth(V_{L}(t))$, as the difference between the highest and lowest powers of $t$ that appear in $V_{L}(t)$. Moreover,  we let $c(L)$ be the crossing number of $L$ and $g_T(L)$ its Turaev genus. It is  already known that a link $L$ is alternating if and only if $g_T(L)=0$. Moreover  in \cite[Corollary\,7.3]{DFKLS}, Dasbach et al. proved that $g_{T}(L) \leq  c(L) -\breadth(V_{L}(t))$ for any non-split link $L$.  Later, Abe \cite{A}  proved that this inequality becomes an equality if the given link is adequate  \cite[Theorem\,3.2]{A}. Further, Kalfagianni and Lee  \cite{KL} proved that  a link is  adequate if  its  Jones diameter $jd_{L}$ is  twice  its crossing number. In this paper, we prove the converse of Abe's result by showing that the Jones diameter of a link that satisfies the above equality is equal to twice of its crossing number. Therefore, a link is adequate if and only if it satisfies  $\breadth(V_{L}(t)) = c(L) -g_{T}(L)$.

\section{Main Result and its Proof}

In this section, we first briefly recall the basic notation and terminology that will be used in the rest of this paper.
\begin{defn}
The Turaev genus $g_{T}(L)$ of the non-split link $L$ is defined as $$g_{T}(L) = \min \bigl\{g_{T}(D)\bigr\} = \min \bigl\{\frac{1}{2}\left(c(D) + 2 - |\sigma_{A}(D)| - |\sigma_{B}(D)|\right) | \ D \ \text{is a diagram of} \ L\bigr\}.$$
\end{defn}
\begin{defn}
The Kauffman bracket polynomial is a function from the set of unoriented link diagrams in the oriented plane to the ring of Laurent polynomials with integer coefficients in an indeterminate $A$. It maps a link diagram $L$ to $\left\langle L\right\rangle\in \mathbb Z[A^{-1},A]$ and is uniquely  determined by the  following relations:
\begin{enumerate}
\item $\left\langle \bigcirc \right\rangle=1$,
\item $\left\langle \bigcirc \cup L\right\rangle=(-A^{-2}-A^2)\left\langle L\right\rangle$,
\item $\left\langle L\right\rangle=A\left\langle L_0\right\rangle+A^{-1}\left\langle L_1\right\rangle$,
\end{enumerate}
where $\bigcirc$ denotes the unknot and  $L,L_0, \text{and } L_1$  represent three unoriented link diagrams which  are identical except in a small region where they  look as in  Figure \ref{figure}.
\end{defn}

For a given oriented link diagram $L$, we let $x(L)$ be the number of negative crossings and $y(L)$  be the number of positive crossings in $L$, see Figure \ref{Diagram1}. The  writhe of $L$ is defined as the integer $w(L) = y(L) - x(L)$.

\begin{defn}
The Jones polynomial $V_{L}(t)$ of an oriented link $L$ is the Laurent polynomial in $t^{1/2}$ with integer coefficients defined by
\begin{equation*}
V_{L}(t)=((-A)^{-3w(L)}\left\langle L \right\rangle)_{t^{1/2}=A^{-2}}\in \mathbb Z[t^{-1/2},t^{1/2}],
\end{equation*}
where $\left\langle L \right\rangle$ denotes the bracket polynomial of the link diagram $L$ with orientation ignored.
\end{defn}

\begin{figure}[h]
	\centering
		\includegraphics[scale=0.10]{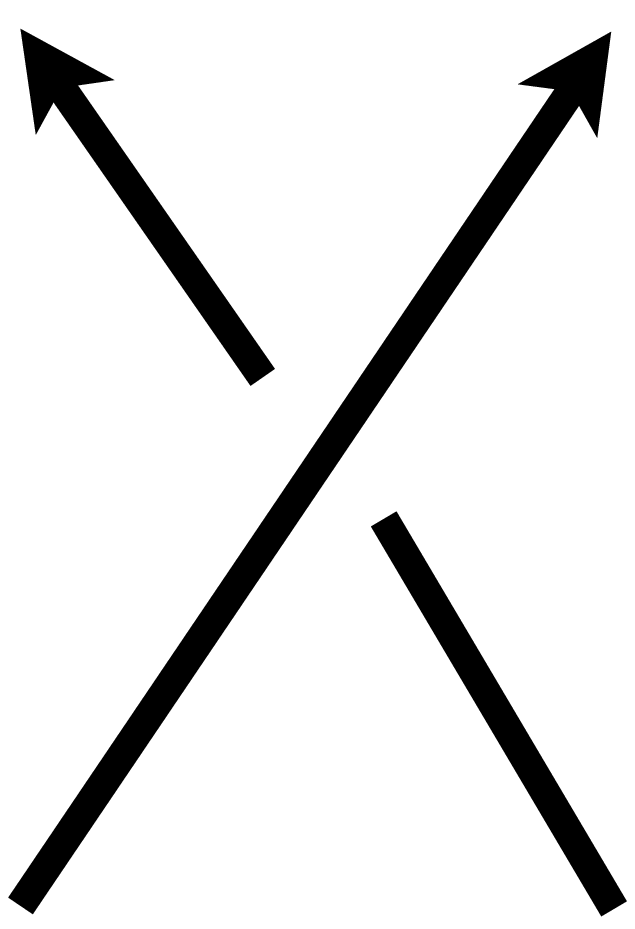}\hspace{1.5cm}\reflectbox{\includegraphics[scale=0.10]{Diagram1}}
	\caption{Positive and negative crossings respectively}
	\label{Diagram1}
\end{figure}

\begin{defn}
The unreduced $n$-th colored Jones polynomial of the link $L$, with a diagram $D$ is given by $J_{L}(n)(t) = ((-1)^{n-1}A^{n^{2}-1})^{w(D)}\langle D^{n-1}\rangle_{A^{-2} = t^{1/2}}$, where 
$D^{n}$ denotes the $n$-blackboard cable of $D$ decorated by the Jones-Wenzl projector $f^{(n)}$. 
\end{defn}

\begin{rem}
It is easy to see that $J_{L}(2)(t) = (-t^{1/2}-t^{-1/2})V_{L}(t)$. For further details on the Kauffman bracket skein modules and the Temperley-Lieb algebra $TL_{n}$, we refer the reader to \cite{KL2,P}.
\end{rem}

Now, we state and prove the main theorem in this paper which  provides  a simple characterization of adequate links through  some numerical link invariants.  Prior to that, we establish and prove   a few  key lemmas that are crucial for the   proof of our   main theorem.
\begin{lemm}\label{basic}
Let $L$ be a non-split link and $D$ a diagram of $L$, then  we have:  $$\breadth(V_{L}(t)) \leq c(D) -g_{T}(D).$$
\end{lemm}
\begin{proof}
For any link diagram $D$ and according to \cite[Lemma\,5.4]{Lick}, we have
 \begin{align*}
 \begin{array}{rl}
\breadth (\langle D\rangle)& \leq 2c(D) +2 |\sigma_{A}(D)| + 2|\sigma_{B}(D)| -4 \\ & \leq 2c(D) + 2c(D) + 4 - 4g_{T}(D) -4 = 4c(D) - 4g_{T}(D),
\end{array}
\end{align*}
 where the second inequality follows from the fact that $|\sigma_{A}(D)| + |\sigma_{B}(D)| \leq c(D) +2$ for any diagram $D$ of the link $L$ and from the fact $2g_{T}(D) = \left(c(D) + 2 - |\sigma_{A}(D)| - |\sigma_{B}(D)|\right)$. Now the result  follows as a consequence of $\breadth (\langle D\rangle) =  4\breadth (V_{L}(t))$.
\end{proof}
In the following remark, we establish  some notations and terminologies for the remainder of this paper.
\begin{rem}\label{notation}
\begin{enumerate}
\item Let $K(S^{2})$ denote the Kauffman bracket skein module of the sphere $S^2$, see \cite{P}.  Any skein element $\mathcal{S} \in K(S^{2})$ decorated by a number of Jones-Wenzl projectors can be expressed as a linear combination of skein elements by expanding  each Jones-Wenzl projector in terms of the elements of the basis $\{1_{n}, e_{1}^{n},e_{2}^{n},\ldots,e_{n-1}^{n}\}$ for the algebra $TL_{n}$ as
\begin{equation}\label{basicequation}
\mathcal{S} = \overline{\mathcal{S}} + \sum f_{*}\overline{\mathcal{S}}_{*},
\end{equation}
where
\begin{enumerate}
\item  $\overline{\mathcal{S}}$ denotes the skein element obtained from $\mathcal{S}$ by replacing each copy of the Jones-Wenzl projector by the identity element $1_{n} \in TL_{n}$.
\item $\overline{\mathcal{S}}_{*}$ denotes the skein element obtained from replacing at least one Jones-Wenzl projector by one of the elements of the set $\{e_{1}^{n},e_{2}^{n},\ldots,e_{n-1}^{n}\}$.
\end{enumerate}
\item For the skein element $\overline{\mathcal{S}} \in K(S^{2})$ as above, we obtain special skein elements from this skein element as follows:
\begin{enumerate}
\item $\overline{\mathcal{S}}_{A}$ denotes the skein element obtained by choosing all-$A$ Kauffman state to all crossings of $\overline{\mathcal{S}} $.
\item $\overline{\mathcal{S}}_{a}$ denotes the skein element obtained by resolving all crossings that correspond to one crossing in the original diagram before the decoration by the Jones-Wenzl projectors in the same way.  We point out that there are $n^{2}$ crossings in  $\overline{\mathcal{S}}$ that correspond to one crossing in the original diagram before decoration. Moreover, there is one-to-one correspondence between such skein elements and the skein elements obtained from the original diagram before decoration after resolving all of its crossings.
\item $\overline{\mathcal{S}}_{e}$ denotes the skein element obtained by resolving all crossings in a way such that this state is not equal to  some $\overline{\mathcal{S}}_{a}$. In other words, this skein element does not correspond to a skein element of the original diagram before decoration after resolving all of its crossings.
\end{enumerate}
\end{enumerate}
\end{rem}
We define  the maximum and minimum degrees of a Kauffman state $\sigma(D)$ as follows: $\maxdeg(\sigma(D)) = a(\sigma(D))-b(\sigma(D)) +2|\sigma(D)|-2$ and $\mindeg(\sigma(D)) = a(\sigma(D))-b(\sigma(D)) - 2|\sigma(D)|+2$, where $a(\sigma(D))$ and $b(\sigma(D))$ are the numbers of the crossings of the diagram $D$ resolved using $A$- and $B$-resolutions, respectively. It follows that $\maxdeg(\sigma(D)) \leq \maxdeg(\sigma_{A}(D))$ and $\mindeg(\sigma_{B}(D)) \leq \mindeg(\sigma(D))$ for any Kauffman state $\sigma(D)$ as a consequence of \cite[Corollary\,3.16]{KL}. In what follows, the term {\it extreme $A$-state} (respectively {\it extreme $B$-state}) refers to a state for which the first (respectively the  second) inequality becomes an equality. Also, we define the {\it parity} of a given state $\sigma(D)$ as the parity of $|\sigma(D)|$.

Given a  rational function $f(A) =\frac{P(A)}{Q(A)}$ with $P(A)$ and $Q(A)$ polynomials with complex coefficients in the variable $A$. Following \cite[Section\,3.4]{KL}, we recall the definition of $\maxdeg(f(A))$ as the maximum power of $A$ in the formal Laurent series expansion of $f(A)$ whose $A$-power is bounded from above. We note that $\maxdeg f(A) = \maxdeg(P(A)) - \maxdeg(Q(A))$. The minimum degree of $f(A)$,  $\mindeg(f(A))$, can be also defined in a similar manner.

\begin{lemm}\label{new1}
Suppose that $\mathcal{S} \in K(S^{2})$ is a skein element decorated by a number of the Jones-Wenzl projectors. Then $\maxdeg \left(f_{*}\langle \overline{\mathcal{S}}_{*}\rangle\right)< \maxdeg \left(A^{c} \langle \overline{\mathcal{S}}_{A} \rangle\right)$, where $\overline{\mathcal{S}}_{*}$ and $\overline{\mathcal{S}}_{A}$ are the skein elements introduced in Remark \ref{notation} and $c$ is the total number of crossings of the diagram of $\overline{\mathcal{S}}$.
\end{lemm}

\begin{proof}
We consider first the case when the diagram $D$ is a crossingless diagram. In other words,  $D$ is a disjoint union of simple closed curves. In this case, it is not too hard to see that the states $\overline{\mathcal{S}}_{*}$ and $\overline{\mathcal{S}}=\overline{\mathcal{S}}_{A}$ consist of simple closed curves and that the number of simple closed curves in the first state is less than the number of simple closed curves in the second one. Moreover, the coefficient $f_{*}$ of any state $\overline{\mathcal{S}}_{*}$ in the linear combination expressing the original state $\mathcal{S}$ in terms of the basis elements $\{1_{n}, e_{1}^{n}, e_{2}^{n},\ldots, e_{n-1}^{n}\}$ in Equation (\ref{basicequation}) has negative maximum degree if the number of Jones-Wenzl projectors is not zero. Thus the result follows directly as a result of combining these two facts. Now we consider the case if $D$ has some crossings. Now as a consequence of \cite[Corollary\,3.16]{KL}, we know that $\maxdeg\left(\langle \mathcal{S} \rangle\right)\leq \maxdeg\left(A^{c}\langle \mathcal{S}_{A} \rangle\right)$ for any skein element $\mathcal{S} \in K(S^{2})$ and for any choice of resolving the crossings of such element. This implies that
\begin{align*}
\maxdeg\left(f_{*}\langle \overline{\mathcal{S}}_{*} \rangle\right)\leq \maxdeg\left(A^{c}f_{*}\langle \overline{\mathcal{S}}_{*A} \rangle\right) &  = c+   \maxdeg\left(f_{*}\langle \overline{\mathcal{S}}_{*A} \rangle\right) \\ &  < c+ \maxdeg\left(\langle \overline{\mathcal{S}}_{A} \rangle\right) = \maxdeg\left(A^{c} \langle \overline{\mathcal{S}}_{A}\rangle\right),
\end{align*}
where the last inequality follows from the case when the skein element has no crossings.
\end{proof}

\begin{rem}\label{new3}
If $\mathcal{S}$ is a skein element in $K(S^{2})$ that is obtained from the original diagram by resolving all of its crossings such that $\maxdeg(\langle \mathcal{S}\rangle) = a(\sigma(D)) - b(\sigma(D)) +2|\sigma(D)|-2$. Then  $\maxdeg(\langle \overline{\mathcal{S}}_{a}\rangle) = a(\sigma(D))n^{2} - b(\sigma(D))n^{2} +2|\sigma(D)|n-2$ where $\overline{\mathcal{S}}_{a}$ is the skein element that corresponds to the original state $\mathcal{S}$ as in Remark \ref{notation}(2)(b). This is true because decorating any diagram with $1_{n}$ gives a  diagram of $n$-parallels of the link. Obviously,  the  number of components of this new link is a  multiple of $n$.  
\end{rem}

\begin{lemm}\label{extreme}
Let $D$ be a link diagram and let $E$ be the link diagram obtained from $D$ by taking the $n$-parallel copies of $D$. For each extreme $A$-state, there are  $(2^{n}-1)^{k}$-extreme $A$-states of $E$, where $k$ is the number of resolved crossings using $B$-resolution in the given extreme $A$-state. Moreover, any extreme $A$-state of $E$ is obtained from this correspondence to some extreme $A$-state of $D$.
\end{lemm}
\begin{proof}
 We note that the state $\overline{\mathcal{S}}_{A}$ is the $n$-parallel copies of the state $\mathcal{S}_{A}$ obtained from the original diagram before decoration. The claim follows easily if the number $k$ to obtain an extreme $A$-state of the the diagram $D$ is zero, therefore we can assume that $k> 0$. Now any state circle in $\mathcal{S}_{A}$ that splits into two state circles after replacing an $A$-resolution by a $B$-resolution at some resolved crossing corresponds to the inner most circle of the $n$-state circles in $\overline{\mathcal{S}}_{A}$ of the diagram $E$ corresponding to such state circle in $\mathcal{S}_{A}$. This state circle has $n$-resolved crossings using $A$-resolution corresponding to the resolved crossing in $\mathcal{S}_{A}$. It is not too hard to see that replacing the $A$-resolution by $B$-resolution of a nonempty collection of these crossings increases the number of state circles by the same number of crossing in this collection. Therefore, we obtain $(2^{n}-1)$-extreme $A$-states of $E$ corresponding to one extreme $A$-state of $D$. We apply induction on $k$ and the fact that any extreme $A$-state of $(k+1)$-resolved crossings using $B$-resolution can be obtained from an extreme $A$-state of $k$-resolved crossings using $B$-resolution and then replace an $A$-resolution by a $B$-resolution at some other resolved crossing using $A$-resolution.

\end{proof}

\begin{lemm}\label{newsimple}
Let $D$ be a link diagram and let $E$ be the link diagram obtained from $D$ by taking the $n$-parallel copies of $D$. Then the absolute value of the coefficient of the extreme monomial $A^{c(D) + 2|\sigma_{A}(D)|-2}$ in $\langle D\rangle$ is equal the absolute value of the coefficient of the extreme monomial $A^{c(E) + 2|\sigma_{A}(E)|-2}$ in $\langle E\rangle$.
\end{lemm}
\begin{proof}
It is easy to see that the coefficient of the extreme monomial $A^{c(D) + 2|\sigma_{A}(D)|}$ in $\langle D\rangle$ is given by $ \sum\limits_{\sigma(D)}(-1)^{|\sigma(D)|-1}$. Now we analyze the correspondence between the extreme $A$-states of $D$ and the extreme $A$-states of $E$ in Lemma \ref{extreme} and their parities. We note that each extreme $A$-state of $D$ corresponds to the odd number  $m = \left(\sum\limits_{i=1}^{n} \binom{n}{i}\right)^{k}$ of extreme $A$-states of $E$ that splits into states of even and odd parity.  Among these states, the difference between the number of even extreme $A$-states and odd extreme $A$-states of $E$ corresponding to some fixed extreme $A$-state of $D$ is plus or minus one. Also, this difference is independent of such extreme $A$-state of $D$ and it depends only on the the parity of $n$. 
Therefore, we obtain 
\begin{align*} 
\big|\sum\limits_{\sigma(E)} (-1)^{|\sigma(E)|-1}\big| & = \big| \sum\limits_{\sigma(D)}\left(\sum\limits_{i=1}^{\frac{m+1}{2}} (-1)^{|\sigma(D)|-1} - \sum\limits_{i=1}^{\frac{m-1}{2}} (-1)^{|\sigma(D)|-1} \right) \big| \\ & = \big|\sum\limits_{\sigma(D)}(-1)^{|\sigma(D)|-1}\big|,
\end{align*}
 where the summations are over the extreme $A$-states of $D$ and $E$.
\end{proof}

Now, we introduce the main result of this paper which states that if the equality in Lemma \ref{basic} holds then the given link $L$ is adequate.

\begin{thm}\label{main}
Let $L$ be a link such that $\breadth(V_{L}(t)) = c(L) -g_{T}(L)$, then $L$ is adequate.
\end{thm}

\begin{proof}
We choose a diagram $D$ of the link $L$ such that $c(L) = c(D)$. According to the result in \cite{T1}, it is necessary and sufficient to show that $D$ is an adequate diagram for $L$ to be adequate. We know that $\breadth (\langle D\rangle) = 4c(D) - 4g_{T}(L)$ for any link diagram $D$ of $L$ since the breadth of the Kauffman bracket is a link invariant. Also, we know that $M= \maxdeg(\langle D\rangle) = c(D) +  2|\sigma_{A}(D)| - 4i -2$, and  $m=\mindeg(\langle D\rangle) = -c(D) -  2|\sigma_{B}(D)| + 4j+2$ for some nonnegative integers $i, j \geq 0$. Hence, we obtain that $M-m =\breadth (\langle D\rangle) = 2c(D) + 2|\sigma_{A}(D)| + 2|\sigma_{B}(D)| -4i -4j-4= 4c(D) - 4g_{T}(L)$ or simply $|\sigma_{A}(D)| + |\sigma_{B}(D)| = c(D) + 2 - 2g_{T}(L) + 2i + 2j$. Also, we have $|\sigma_{A}(D)| + |\sigma_{B}(D)| \leq  c(D) + 2 - 2g_{T}(L)$ from the fact that $g_{T}(L) \leq g_{T}(D) = \frac{1}{2}\left(c(D)+2-|\sigma_{A}(D)|-|\sigma_{B}(D)|\right)$. This implies that $i = j = 0$. Thus, we conclude that the monomials $A^{c(D)+2|\sigma_{A}(D)|-2}$ and $A^{-c(D)-2|\sigma_{B}(D)|+2}$ have nonzero coefficients in $ \langle D\rangle$.  As a consequence of Lemma \ref{newsimple}, Lemma \ref{new1} and Remark \ref{new3}, we obtain $\maxdeg(\langle D^{n}\rangle) = \maxdeg(\langle E \rangle) =c(E) + 2|\sigma_{A}(E)| -2 =c(D)n^{2} + 2|\sigma_{A}(D)|n-2$, where $E$ is the link diagram obtained from $D$ by taking $n$-parallel copies of $D$. By applying a similar argument on the mirror image of $D$, we also conclude that $\mindeg(\langle D^{n}\rangle) = \mindeg(\langle E \rangle) = -c(D)n^{2} - 2|\sigma_{B}(D)|n+2$. Therefore, we obtain
\[ \maxdeg(J_{L}(n)(t)) = c(D)\frac{(n-1)^{2}}{4} + |\sigma_{A}(D)|\frac{(n-1)}{2} + w(D)\dfrac{n^{2}-1}{4} - \frac{1}{2},\]
 and
 \[ \mindeg(J_{L}(n)(t)) = -c(D)\frac{(n-1)^{2}}{4} - |\sigma_{B}(D)|\frac{(n-1)}{2} + w(D)\dfrac{n^{2}-1}{4} + \frac{1}{2}.\]

Now, we can easily see that the set
\begin{align*}
\Bigl\{4n^{-2}\Big(\maxdeg(J_{L}(n)(t)) - & \mindeg(J_{L}(n)(t))\Bigr)\Bigr\}  \\ = & \Bigl\{2c(D)\frac{(n-1)^{2}}{n^{2}} +2|\sigma_{A}(D)|\frac{(n-1)}{n^{2}}+2|\sigma_{B}(D)|\frac{(n-1)}{n^{2}} -\frac{4}{n^{2}}\Bigr\}
\end{align*}
 has a single cluster point, which is  $2c(D)$. Thus, we get $jd_{L} = 2c(D) =2c(L)$ and this proves that $L$ is adequate according to the characterization of adequate links given in \cite[Theorem\,1.1]{KL}.

\end{proof}
\begin{rem}
Combining Theorem \ref{main} and the result of \cite[Theorem\,3.2]{A}, we conclude  that the condition $\breadth(V_{L}(t)) = c(L) -g_{T}(L)$ is a necessary and sufficient condition for a link to be adequate.
\end{rem}

\begin{rem}
It is worth mentioning here   that the main result in Theorem \ref{main} is motivated by the work of Lickorish and  Thistlethwaite in  \cite{LT}. This result along with some of its  consequences discussed in  the next section can be considered as a natural generalization of certain  results in \cite{LT}.
\end{rem}

Now, let us illustrate the use of Theorem \ref{main} by an example.
\begin{ex}\label{example}
The 10-crossing knots $10_{145}$ and $10_{161}$ have breadth of the Jones polynomial  equal to eight and Turaev genus one \cite{knotinfo}. So they do not satisfy the condition of Theorem \ref{main}.  In conclusion, the knots $10_{145}$ and $10_{161}$ are  not adequate.
\end{ex}

\section{Consequences and Applications}
In this section, we shall discuss some consequences and applications of our main result.

\begin{coro}
Let $L$ be a link such that $\breadth(V_{L}(t)) = c(L) -g_{T}(L)$, then $jd_{L} = 2c(L)$.
\end{coro}

\begin{coro}\label{main2}
Let $L$ be a link such that $\breadth(V_{L}(t)) = c(L)- 1 - g_{T}(L)$, then $L$ is semi-adequate.
\end{coro}

\begin{proof}
Assume that $D$ is a diagram of $L$ with $c(L) = c(D)$. It is easy to see that $\breadth (\langle D\rangle) = 4c(D)-4g_{T}(L)-4, M= \maxdeg(\langle D\rangle) = c(D) +  2 \sigma_{A}(D)| - 4i-2$ and  $m=\mindeg(\langle D\rangle) = -c(D) -  2|\sigma_{B}(D)| + 4j+2$ for some nonnegative integers $i, j \geq 0$.  Thus, we obtain: \[M-m =\breadth (\langle D\rangle) = 2c(D) + 2|\sigma_{A}(D)| + 2|\sigma_{B}(D)| -4i -4j -4 = 4c(D)-4g_{T}(L) -4.\]

This implies that $|\sigma_{A}(D)| + |\sigma_{B}(D)| = c(D) + 2i + 2j-2g_{T}(L)$. From the fact that $g_{T}(L) \leq g_{T}(D) = \frac{1}{2}\left(c+2-|\sigma_{A}(D)|-|\sigma_{B}(D)|\right)$, we get $|\sigma_{A}(D)| + |\sigma_{B}(D)| \leq  c(D) + 2 - 2g_{T}(L)$. Hence, we conclude that $2i+2j=2$ or simply that either $i=1$ or $j=1$ but not both. Therefore, either the monomial $A^{c(D)+2|\sigma_{A}(D)|-2}$ or the monomial $A^{-c(D)-2|\sigma_{B}(D)|+2}$ has nonzero coefficient in $ \langle D\rangle$. As a consequence of Lemma \ref{newsimple}, Lemma \ref{new1} and Remark \ref{new3}, we obtain $\maxdeg(\langle D^{n}\rangle) = \maxdeg(\langle E \rangle) =c(E) + 2|\sigma_{A}(E)| -2 =c(D)n^{2} + 2|\sigma_{A}(D)|n-2$, or $\mindeg(\langle D^{n}\rangle) = \mindeg(\langle E \rangle) =c(E) - 2|\sigma_{B}(E)| +2 =-c(D)n^{2} - 2|\sigma_{B}(D)|n+2$, where $E$ is the link diagram obtained from $D$ by taking $n$-parallel copies of $D$. Therefore, we obtain either  \[ \maxdeg(J_{L}(n)(t)) = c(D)\frac{(n-1)^{2}}{4} + |\sigma_{A}(D)|\frac{(n-1)}{2} + w(D)\dfrac{n^{2}-1}{4} - \frac{1}{2} = \frac{y(D)}{2}n^{2} + \mathcal{O}(n),\]  or  \[ \mindeg(J_{L}(n)(t)) = -c(D)\frac{(n-1)^{2}}{4} - |\sigma_{B}(D)|\frac{(n-1)}{2} + w(D)\dfrac{n^{2}-1}{4} + \frac{1}{2} = -\frac{x(D)}{2}n^{2} + \mathcal{O}(n),\]  where $c(D) = y(D) + x(D)$ and $w(D) = y(D) - x(D)$. Thus the result follows as a consequence of \cite[Theorem\,2.4]{KL}
\end{proof}

\begin{rem}
We point out that the condition $\breadth(V_{L}(t)) = c(L)- 1 - g_{T}(L)$ does not characterize semi-adequate links. In particular, the knot $8_{19}$ is known to be semi-adequate but $\breadth(V_{L}(t)) = 5 \neq 6 = c(L) - 1 - g_{T}(L)$.
\end{rem}

\begin{coro}
Let $L$ be a non-split adequate link. Then every diagram $D$  of $L$ satisfies $|\sigma_{A}(D)|+|\sigma_{B}(D)| \leq c(D) + 2 - 2g_{T}(L)$ and the equality occurs when $c(D) = c(L)$.
\end{coro}

\begin{coro}\label{main1}
Let $L$ be a link with $\breadth(V_{L}(t)) = c(L) -1$. Then $L$ is an adequate non-alternating link.
\end{coro}

\begin{proof}
Note that $L$ is non-alternating since $\breadth(V_{L}(t)) < c(L)$. Using the result of \cite[Corollary\,7.3]{DFKLS} that states  $g_{T}(L) \leq  c(L) -\breadth(V_{L}(t)) $ for any link $L$, we get $g_{T}(L)=1$. Thus, we obtain that the link is adequate according to the result of Theorem \ref{main}.
\end{proof}

Another interesting generalization of alternating links is the class of quasi-alternating links. These links  have been introduced  by Ozsv$\acute{\text{a}}$th and Szab$\acute{\text{o}}$ while studying the Heegaard Floer homology of  branched double covers along alternating  links   \cite{OS}. Quasi-alternating links are known to be thin in Khovanov homology \cite{MO}.  Many obstructions have been proved  recently for a link to be quasi-alternating. In particular, obstructions using the Jones polynomials have been studied in \cite{CQ,QC}. The following corollary introduces another simple obstruction.

\begin{coro}
Let $L$ be a link such that $\breadth(V_{L}(t)) = c(L) - g_{T}(L)$ with $g_{T}(L) > 0$, then $L$ is not quasi-alternating. In particular, any semi-alternating link is not quasi-alternating.
\end{coro}

\begin{proof}
As a consequence of Theorem \ref{main}, the link $L$ is adequate. The  assumption that $g_{T}(L) > 0$ implies that  $L$ is not alternating.  Recall that a nonalternating adequate link is homologically thick in Khovanov homology  \cite[Proposition\,5.1]{Kh}. Therefore, $L$ is not quasi-alternating since quasi-alternating links are  homologically thin in Khovanov homology  \cite[Theorem\,1]{MO}. The last statement follows from the fact that the breadth of the Jones polynomial of any semi-alternating links is equal to its crossing number minus one according to \cite[Proposition\,5]{LT}.
\end{proof}

\begin{rem}
We point out that semi-alternating links were defined for the first time in \cite[Section\,2]{LT} and known to have non-alternating Jones polynomial according to  \cite[Proposition\,5]{LT} and this gives another proof that they are not quasi-alternating.
\end{rem}

\begin{coro}
Let $L$ be a quasi-alternating link. Then either $c(L) = \breadth(V_{L}(t))$ or $ \breadth(V_{L}(t)) \leq c(L)-2$.
\end{coro}

\begin{proof}
The proof is straightforward using Corollary \ref{main1} and the fact that  $c(L) = \breadth(V_{L}(t))$ if and only if $L$ is alternating.
\end{proof}

The following result  is an analogue of  \cite[Corollary\,5.1]{KL}:
\begin{coro}\label{simple}
Suppose $L$ is a non-adequate link with a diagram $D$ such that $\breadth(V_{L}(t)) = c(D) - l$ for $ l \geq 2$, then $c(D) - l + 2 \leq c(L) \leq c(D)$.
\end{coro}

\begin{proof}
We know that $\breadth(V_{L}(t)) \leq c(L)$ and this implies that $c(D) - l \leq c(L)$. Now suppose that $c(L) < c(D)$, then there is another diagram $D^{'}$ of $L$ such that $c(L) = c(D^{'})< c(D)$. We have that $c(D^{'}) \geq \breadth(V_{L}(t)) = c(D) - l$. If the equality occurs, then $L$ is alternating which implies that it is adequate. 
Now if $c(D^{'}) = c(D) - l + 1$, then $L$ is adequate since $\breadth(V_{L}(t)) = c(L)-1$. The last two cases are not plausible since the given link is not adequate. Thus we obtain $c(D) - l + 2 \leq c(D^{'}) = c(L) \leq c(D)$.
\end{proof}

As a direct application, the above corollary can be used to establish a special case of the conjecture about the crossing number of the connected sum of links.
\begin{coro}
Let $D$ be a diagram of a knot $K$ such that $\breadth(V_{K}(t)) = c(D) -2$ and let  $K_{1},K_{2},\ldots, K_{m}$ be a family of  alternating knots, then
 $c(K\#K_{1}\#K_{2}\#\ldots\#K_{m}) = c(K) + \sum\limits_{i=1}^{m} c(K_{i}).$
\end{coro}
\begin{proof}
The result follows directly from the additivity of the breadth of the Jones polynomial under the  connected sum operation and Corollary \ref{simple}.
\end{proof}

\begin{rem}
\begin{enumerate}
\item The result in Corollary \ref{simple} with $l=2$ can be used to confirm the computation of the crossing number of 36 non-alternating knots of 10 crossings or less. Some of these knots are $8_{20}, 8_{21}, 9_{43}, 9_{44}, \ldots$. These knots satisfy that the crossing number equals to the breadth of its Jones polynomial plus two.
\item There are at least 96 non-alternating knots of 12 crossings or less that satisfy the condition in Corollary \ref{main1}. This proves that all these knots are adequate. Some of these knots are $10_{152},10_{153},10_{154}, 11_{n6}, 11_{n9}, \ldots$.
\item  Corollary \ref{main1} does not hold if $\breadth(L) = c-2$. The knots $10_{145}$ and $10_{161}$ have breadth equal to eight which is the crossing number minus two. These knots are not adequate as it was discussed  in Example \ref{example}.
\end{enumerate}
\end{rem}

\end{document}